\documentclass[a4paper,11pt]{amsart}
\usepackage{amsmath}
\usepackage{cases}
\usepackage{amsfonts}
\usepackage[colorlinks,linkcolor=blue,citecolor=blue]{hyperref}
\usepackage{latexsym, amssymb, amsmath, amsthm, bbm}
\usepackage[all]{xy}
\usepackage{pgfplots}
\usepackage{mathrsfs}
\usepackage [latin1]{inputenc}
\usepackage[misc]{ifsym}

\DeclareSymbolFont{EulerExtension}{U}{euex}{m}{n}
\DeclareMathSymbol{\euintop}{\mathop} {EulerExtension}{"52}
\DeclareMathSymbol{\euointop}{\mathop} {EulerExtension}{"48}

\allowdisplaybreaks[4]

\def \id{\operatorname{Id}}
\def \ker{\operatorname{Ker}}

\def \dim{\operatorname{dim}}

\def \Hom{\operatorname{Hom}}

\def \Id{\operatorname{Id}}

\def \id{\operatorname{Id}}
\def \ker{\operatorname{Ker}}

\numberwithin{equation}{section}

\newtheorem{theorem}{Theorem}[section]
\newtheorem{lemma}[theorem]{Lemma}
\newtheorem{proposition}[theorem]{Proposition}
\newtheorem{corollary}[theorem]{Corollary}
\newtheorem{definition}[theorem]{Definition}
\newtheorem{example}[theorem]{Example}
\newtheorem{remark}[theorem]{Remark}

\newtheorem{convention}[theorem]{Convention}

\begin{document}
\title{Construction of factorizable Hopf algebras}
\thanks{Kun Zhou: Yanqi Lake Beijing Institute of Mathematical Sciences and Applications, Beijing 101408, China. email: kzhou@bimsa.cn}


\author{Kun Zhou}
\date{}
\maketitle
\begin{abstract}
We focus on the problem of producing new modular tensor categories from Hopf algebras. To do this, we first give a general method to construct factorizable Hopf algebras. Then we apply the method to construct two families of ribbon factorizable Hopf algebras which are not quantum groups or Drinfel'd double. One family are point Hopf algebras which give non-semisimple modular tensor categories, while the other family are $C^*$-Hopf algebras which give new unitary modular tensor categories. Lastly, we determine the fusion rings of these unitary modular tensor categories.
\end{abstract}

\section{Introduction}
A modular tensor category (MTC) is a braided finite category with some additional algebraic structures (duality, twist, and a non-degeneracy axiom, see \cite{Ki, Ker}). It provides a topological quantum field theory in dimension 3, and in particular, invariants of links and 3-manifolds (see \cite{NV, Ker}). In the past few decades, an application of unitary MTCs to quantum computing has been
proposed by Freedman and Kitaev and advanced in the series of papers (see \cite{Fre, Fre1}). So MTCs are intensively studied in recent years. MTCs can be divided into semisimpe cases and non-sesimple cases. Both of them are studied by many authors (such as \cite{Ker, Shi, Ro, Mug, zheng}). It's known that the notion of a MTC may be thought of as a categorical generalization of a ribbon factorizable Hopf algebra. More precisely, semisimple MTCs may be viewed as a categorical generalization of semisimple ribbon factorizable Hopf algebras while non-semisimple ones may be thought as categorical generalization of a non-semisimple ribbon factorizable Hopf algebras. Many open problems about MTCs and the classification of MTCs make the construction of MTCs to be very important. There are many authors who have studied the construction of MTCs (such as \cite{Ro, Bla, Gui}). To our best knowledge, the known MTCs which are arising from Hopf algebras come from Drinfel'd double or small quantum groups. Therefore, constructing MTCs which are not from Drinfel'd double or small quantum groups may provide a new idea for the study of MTCs. To achieve this, we first give a general method to obtain factorizable Hopf algebras. Using the method, we construct a large number of factorizable ribbon Hopf algebras which are not quantum groups or Drinfel'd double. In particular, both semisimple MTCs(especially unitary MTCs) and non-semisimple MTCs are constructed.

Our idea about construction of factorizable Hopf algebras is as follows. Given a finite dimensional Hopf algebra $H$, we consider quotients of $D(H)$ which make them to be factorizable Hopf algebras. The difficult part is how to choose Hopf ideal $I$(resp. Hopf $\ast$ ideal $I$) of $D(H)$. To overcome the difficult, we consider central group-like elements of $D(H)$ and use extension of Hopf algebras to achieve this. It's known that every factorizable Hopf algebra is quotient of some Drinfel'd double, hence our construction method is natural in this sense.

The paper is organized as follows. Section 2 is devoted to give some notation and preliminary results. Then we give a general method to obtain factorizable Hopf algebras in Section 3. Section 4 is divided into two subsections. In subsection 4.1, we apply the method to construct two families of ribbon factorizable Hopf algebras. In subsection 4.2, we determine fusion rings of a family of new unitary modular tensor categories which are arising from $C^*$-Hopf algebras.

\begin{convention}\emph{Throughout the paper we work over an algebraically closed field $\Bbbk$ of characteristic 0 if not specified. All Hopf algebras in this paper are finite dimensional. For the symbol $\delta$, we mean the classical Kronecker's symbol. Our references for the theory of Hopf algebras (resp. tensor category) are \cite{MS,R} (resp. \cite{Ki}). For a Hopf algebra $H$, the antipode of $H$ will denoted by $S$. For a Hopf algebra $H$, the group of group-like elements in $H$ will be denoted by $G(H)$.}
\end{convention}

\section{Preliminaries}
We collect some necessary notions and results in this section.
\subsection{Modular tensor category and Factorizable Hopf algebra.}
A modular tensor category is a ribbon and finite abelian $\Bbbk$-linear category $\mathcal{C}$ satisfying a non-degeneracy condition, i.e. if $c_{V,W}\circ c_{W,V}=\Id_{W\otimes V}$ for arbitrary object $W$ then $V\cong \mathbb{I}^n$ for some $n\in \mathbb{N}$ (see \cite{Shi, Ker}). If in addition $\mathcal{C}$ is semsimple and unitary, here unitary means that there is conjugation $*:\mathcal{C}\rightarrow \mathcal{C}$ such that the Hermitian form $(f, g) = tr(fg^*)$ is positive definite on $\Hom(X, Y )$ for any two objects $X, Y\in \mathcal{C}$, then $\mathcal{C}$ is called unitary modular tensor category (see \cite{Ro}).

Factorizable Hopf algebras are special class of quasitriangular Hopf algebras. Recall that a quasitriangular Hopf algebra is a pair $(H, R)$ where $H$ is a Hopf algebra over $\Bbbk$ and $R=\sum R^{(1)} \otimes R^{(2)}$ is an invertible element in $H\otimes H$ such that
\begin{equation*}
 (\Delta \otimes \id)(R)=R_{13}R_{23},\; (\id \otimes \Delta)(R)=R_{13}R_{12},\;\Delta^{op}(h)R=R\Delta(h),
 \end{equation*}
for $h\in H$. Here by definition $R_{12}= \sum R^{(1)} \otimes R^{(2)}\otimes 1 $ and similarly for $R_{13}$ and $R_{23}$. For a quasitriangular Hopf algebra $(H,R)$, there are Hopf algebra maps $f_{R_{21}R}: H^{\ast cop}\rightarrow H$ and $g_{R_{21}R}: H^{\ast op}\rightarrow H$, given respectively by
$$f_{R_{21}R}(a):=(a\otimes \Id)(R_{21}R),\;\;g_{R_{21}R}(a):=(\Id\otimes a)(R_{21}R),\;f\in H^\ast.$$
A factorizable Hopf algebra is a quasitriangular Hopf algebra $(H, R)$ such that $f_{R_{21}R}$, or equivalently $g_{R_{21}R}$, is a linear isomorphism.

\subsection{Hopf exact sequence and Drinfel'd double.}
\begin{definition}\label{def2.1.1}
A short exact sequence of Hopf algebras is a sequence of Hopf algebras
and Hopf algebra maps
\begin{equation}\label{ext}
\;\; K\xrightarrow{\iota} H \xrightarrow{\pi} \overline{H}
\end{equation}
such that
\begin{itemize}
  \item[(i)] $\iota$ is injective,
  \item[(ii)]  $\pi$ is surjective,
  \item[(iii)] $\ker(\pi)= HK^+$, $K^+$ is the kernel of the counit of $K$.
\end{itemize}
\end{definition}

Take an exact sequence \eqref{ext}, then $K$ is a normal Hopf
subalgebra of $H$. Conversely, if $K$ is a normal Hopf subalgebra of a Hopf algebra
$H$, then the quotient coalgebra $H=H/HK^+=H/K^+H$ is a quotient Hopf algebra
and $H$ fits into an extension \eqref{ext}, where $\iota$ and $\pi$ are the canonical maps. If $H$ fits into an extension \eqref{ext} and $H$ is finite dimensional, then $\dim(H)=\dim(K)\dim(\overline{H})$ by the well known "normal basis" theorem for sub-Hopf algebras (see \cite{Sch}).

An extension \eqref{ext} above such that $K$ is commutative and $\overline{H}$ is cocommutative is called \emph{abelian}. In this situation, we know the extension \eqref{ext} can be written in the following form:
\begin{equation*}
\;\; \Bbbk^G\xrightarrow{\iota} H \xrightarrow{\pi} \Bbbk F,
\end{equation*}
where $G, F$ are finite groups. Abelian extensions were classified by Masuoka
(see \cite[Proposition 1.5]{M3}), and the above $H$ can be expressed as $\Bbbk^G\#_{\sigma,\tau}\Bbbk F$.
To give the description of $\Bbbk^G\#_{\sigma,\tau}\Bbbk F$, we need the following data
\begin{itemize}
\item[(i)] A matched pair of groups, i.e. a quadruple $(F,G,\triangleleft,\triangleright)$, where $G\stackrel{\triangleleft}{\leftarrow}G\times F \stackrel{\triangleright }{\rightarrow}F$ are action of groups on sets, satisfying the following conditions
    \begin{align*}
    g\triangleright(ff')=(g\triangleright f)((g\triangleleft f)\triangleright f'),\quad (gg')\triangleleft f=(g\triangleleft(g'\triangleright f))(g'\triangleleft f),
    \end{align*}
    for $g,g'\in G$ and $f,f'\in F$.
\item[(ii)] $\sigma:G\times F\times F\rightarrow \Bbbk^\times$ is a map such that
\begin{align*}
    \sigma(g\triangleleft f,f',f'')\sigma(g,f,f'f'')=\sigma(g ,f,f')\sigma(g,ff',f'')
    \end{align*}
    and $\sigma(1,f,f')=\sigma(g,1,f')=\sigma(g,f,1)=1$, for $g\in G$ and $f,f',f''\in F$.
\item[(iii)] $\tau:G\times G \times F \rightarrow \Bbbk^\times$ is a map satisfying
    \begin{align*}
    \tau(gg',g'',f)\tau(g,g',g''\triangleright f)=\tau(g',g'',f)\tau(g,g'g'',f)
    \end{align*}
    and $\tau(g,g',1)=\tau(g,1,f)=\tau(1,g',f)$, for $g,g',g''\in G$ and $f\in F$. Moreover, the $\sigma, \tau$ satisfy the following compatible condition
    \begin{align*}
    \sigma(gg',f,f')\tau(g,g',ff')&=\sigma(g,g'\triangleright f, (g'\triangleleft f)\triangleright f')\sigma(g',f,f')\\
    &\tau(g,g',f)\tau(g\triangleleft (g'\triangleleft f),g'\triangleleft f,f'),
    \end{align*}
    for $g,g',g''\in G$ and $f,f',f''\in F$.
\end{itemize}

\begin{definition}\cite[Section 2.2]{AA}\label{def2.1.2}
The Hopf algebra $\Bbbk^G\#_{\sigma,\tau}\Bbbk F$ is equal to $\Bbbk^G\otimes \Bbbk F$ as vector space and we write $a\otimes x$ as $a\#x$. The product, coproduct are given by
\begin{align*}
 &(e_g\#f).(e_{g'}\#f')=\delta_{g\triangleleft f,g'}\;\sigma(g,f,f')\;e_g\#(ff'),\\
 &\Delta(e_g\#f)=\sum_{g'g''=g}\;\tau(g',g'',f)\;e_{g'}\#g''\triangleright f\otimes e_{g''}\#f,
\end{align*}
The unit is $\sum_{g\in G}e_g\#1$ and the counit is $\epsilon(e_g\#f)=\delta_{g,1}$ and the antipode is
\begin{align*}
 S(e_g\#f)=\sigma(g^{-1},g\triangleright f,(g\triangleright f)^{-1})^{-1}\;
 \tau(g^{-1},g,f)^{-1}\;e_{(g\triangleleft f)^{-1}}\#(g\triangleright f)^{-1}.
\end{align*}
\end{definition}
Recall that a $C^*$-Hopf algebra is a semisimple $*$-Hopf algebra $H$  over $\mathbb{C}$ such
that $H$ with its underlying $*$-algebra structure is a $C^*$-algebra. In addition if $(H,R)$ is factorizable Hopf algebra, then the tensor category of $\ast$-representation of $H$ on finite dimensional Hilbert spaces is a unitary MTC. For an abelian extension $\mathbb{C}^G\#_{\sigma,\tau}\mathbb{C} F$, it's known that $\mathbb{C}^G\#_{\sigma,\tau}\mathbb{C} F$ is $C^*$-Hopf algebra with following involution
$$(e_g\#f)^*=\sigma(g,f,f^{-1})e_{g\triangleleft f}\#f^{-1}$$
if $|\sigma|=1$ and $|\tau|=1$ (\cite[Theorem 3]{Kac1}). The following example will be used to construct unitary MTCs.
\begin{example}\label{ex2.1.x}
\emph{Let $n\in \mathbb{N}$. A Hopf algebra $H$ belonging to $\Bbbk^G\#_{\sigma,\tau}\Bbbk F$ is denoted by $A(G,\sigma,n)$ if the data $(\triangleright,\tau)$ of $H$ is trivial, i.e. $g\triangleright f=1$ and $\tau=1$ where $g\in G, f\in F$ and the following conditions hold:
  \begin{itemize}
  \item[(i)] $F=\mathbb{Z}_{n}=\langle x|\;x^n=1\rangle$;
       \item[(ii)] there is $b\in G$ with order $n$ such that $g\triangleleft x=bgb^{-1}$.
  \end{itemize}}
\end{example}
A special case of $A(G,\sigma,n)$ is as follows:
\begin{example}\label{ex2.1.5}
\emph{Let $p, q$ be two odd prime numbers such that $p\equiv 1(\text{mod}\;q)$ and let $\omega$ be a primitive $q$th root of 1 in $\Bbbk$. Assume $t\in \mathbb{N}$ satisfying $t^q\equiv 1(\text{mod}\;p)$ and $t \not \equiv 1(\text{mod}\;p)$. Let $0\leq l\leq (q-1)$, then the Hopf algebra $\mathscr{A}_l$ \cite[Lemma 1.3.9]{Na2} belongs to $\Bbbk^G\#_{\sigma,\tau}\Bbbk F$. By definition, the data $(G,F,\triangleleft,\triangleright,\sigma,\tau)$ of $\mathscr{A}_l$ is given by the following way
\begin{itemize}
             \item[(i)] $G=\mathbb{Z}_p\rtimes \mathbb{Z}_q=\langle a,b|\;a^p=b^q=1,bab^{-1}=a^t\rangle$,\;$F=\mathbb{Z}_q=\langle g|\;g^q=1\rangle$. The action $\triangleright $ is trivial, and $a\triangleleft g^{i}=a^{t^i},b\triangleleft g^i=b$, for $0\leq i \leq q-1$.
              \item[(ii)] $\sigma(a^i b^j,g^m, g^n)=w^{jlq_{mn}}$, where $q_{mn}$ is the quotient of $m+n$ in the division by $q$ and $1 \leq i\leq p-1$, $0\leq j,m,n \leq q-1$.
              \item[(iii)] $\tau(g ,g',f)=1$ for $g,g'\in G$ and $f\in F$.
\end{itemize}}
\end{example}

To construct non-semisimple MTCs, we recall the definition of $n$-rank Taft algebra which is given in \cite[Section 5]{Hu}. Assume $n\in \mathbb{N}^*$ and let $M=\{(i,j)|\;1\leq i,j\leq n\}$. Let $\theta:M\times M\longrightarrow \Bbbk$ be the map which is defined by $$\theta(i,j)=\left\{
\begin{array}{cr}
q & i>j \\
1 & i=j\\
q^{-1} & i<j
\end{array} \right. .$$
Assume $q$ is a primitive $l$-th root of unity. Then the $n$-rank Taft algebra $\overline{\mathscr{A}}_q(n)$ is generated by $x_1,...,x_n$, and $g_1,...,g_n$ as an algebra, with the relations
$$g_ig_j=g_jg_i, \;g_i^{l}=1,\;g_ix_j=\theta(i,j)q^{\delta_{i,j}}x_jg_i,\;x_ix_j=\theta(i,j)x_jx_i, \;x_i^l=0.$$
The coproduct, counit and antipode are given by
 \begin{align*}
\Delta(g_i) = g_i\otimes g_i, \;\Delta(x_i) = x_i \otimes 1+g_i \otimes x_i,\\
\epsilon(g_i)=1, \;\epsilon(x_i)=0,\\
S(g_i)=g_i^{-1}, S(x_i) = -g_i^{-1}x_i,
\end{align*}
where $1\leq i,j \leq n$. When $n=1$, $\overline{\mathscr{A}}_q(1)$ is the Taft algebra.

Let $H$ be a Hopf algebra over $\Bbbk$. Recall the definition of Drinfel'd double of $H$ which is $D(H)=(H^*)^{cop}\otimes H$ as coalgebra. The multiplication of $D(H)$ is given by $(f\otimes h)(g\otimes k)=f [h_{(1)}\rightharpoonup g\leftharpoonup S^{-1}(h_{(3)})]\otimes h_{(2)}k$, where $f,g\in H^*, \;h,k\in H$ and $\langle a\rightharpoonup g\leftharpoonup b,c\rangle=\langle g,bca\rangle$ for $a,b,c\in H$. For convenience, we write $fh$ as $f\otimes h$ in the following content.  Let $\mathcal{R}$ be the standard universal $\mathcal{R}$-matrix of $D(H)$. The following result is shown in \cite[Theorem 3]{Kau}.
\begin{theorem}\label{thm0.x}
Let $g$ and $\alpha$ be the distinguished grouplike elements of $H$ and $H^*$, respectively. Then $(D(H), \mathcal{R})$ has a ribbon element if and only if there are $a\in G(H)$ and $\beta\in G(H^*)$ such that
\begin{itemize}
 \item[(i)] $a^2=g$ and $\beta^2=\alpha$;
  \item[(ii)] $S^2(h)=a(\beta\rightharpoonup h \leftharpoonup \beta^{-1})a^{-1}$,\; $h\in H$.
\end{itemize}
\end{theorem}

\section{A way to construct factorizable Hopf algebras}\label{sec2.2}
This section is devoted to give a general method of construction of factorizable Hopf algebras. Assume $H$ is finite dimensional Hopf algebra in the following content. If $g\in G(H)$, we will write $|g|$ as the order of $g$ and denote $\langle g\rangle$ as the group algebra generated by $g$. Let $S_n$ be the symmetric group of degree $n$. Suppose $(H,R)$ is a quasitriangular Hopf algebra and $\pi:H\rightarrow K$ is surjective Hopf map. Then $(K,(\pi\otimes \pi)(R))$ is also a quasitriangular Hopf algebra. In particular, if $I$ is Hopf ideal of $H$ then $(H/I,\overline{R})$ is also quasitriangular Hopf algebra, where $\overline{R}$ is defined through the natural quotient map $\pi(h)=h+I$ for $h\in H$. For convenience, we denote $K^+$ as $\ker\epsilon$ for a Hopf algebra $K$. The following theorem is main result in the section.

\begin{theorem}\label{thm1.x}
Assume $n$ is odd number and $x\in G(H)$ with order $n$. If $\{a_i|\;1\leq i\leq m\}\subseteq H$ and $\chi\in G(H^*),\;\sigma,\tau\in S_m$ such that
\begin{itemize}
 \item[(i)] $xa_i=a_{\sigma(i)}x$ and $\{a_ix^j|\;1\leq i\leq m,\;1\leq j\leq n\}$ is linear basis of $H$;
  \item[(ii)] $a_i\leftharpoonup \chi=a_{\tau^{-1}(i)}$ and $\chi\rightharpoonup a_i=a_{\tau^{-1}\circ \sigma^{-1}(i)}$;
  \item[(iii)] $\chi(x)$ is primitive $n$-th root of unity and $|\chi|$=$n$;
\end{itemize}
then $\langle \chi x\rangle\subseteq D(H)$ is normal subHopf algebra and $(D(H)/I,\overline{\mathcal{R}})$ is factorizable Hopf algebra, where $I=D(H)\langle \chi x\rangle^{+}$ and $\mathcal{R}$ is standard universal $\mathcal{R}$-matrix of $D(H)$.
\end{theorem}
We will introduce the following lemmas to show above theorem.
\begin{lemma}\label{lem3.0.x}
Assume $H$ satisfies the conditions of Theorem \ref{thm1.x}. Denote $\{E_{a_ix^j}|\;1\leq i\leq m,\;1\leq j\leq n\}$ as the dual basis of $\{a_ix^j|\;1\leq i\leq m,\;1\leq j\leq n\}$. Then the following equations hold:
\begin{itemize}
  \item[(i)] $\sigma \circ\tau=\tau \circ\sigma$,
  \item[(ii)]  $\chi E_{a_ix^j}=\chi(x)^jE_{a_{\tau(i)}x^j}$,
  \item[(iii)] $E_{a_ix^j}\chi=\chi(x)^jE_{a_{\sigma\circ\tau(i)}x^j}$,
\end{itemize}
where $1\leq i\leq m,\;1\leq j\leq n$.
\end{lemma}

\begin{proof}
Since $(\chi\rightharpoonup a_i)\leftharpoonup \chi=\chi\rightharpoonup (a_i\leftharpoonup \chi)$, we know (i). Directly, we have
\begin{align*}
\chi E_{a_ix^j}(a_kx^l)&=E_{a_ix^j}[(a_kx^l)\leftharpoonup \chi]\\
                &=E_{a_ix^j}[(a_k\leftharpoonup \chi) x^l]\chi(x)^l\\
                &=\chi(x)^j\delta_{\tau(i),k}\delta_{j,l},
\end{align*}
we have (ii). Since
\begin{align*}
E_{a_ix^j}\chi&=E_{a_ix^j}[\chi\rightharpoonup (a_kx^l)]\\
                &=E_{a_ix^j}[(\chi \rightharpoonup a_k) x^l]\chi(x)^l\\
                &=\chi(x)^j\delta_{\sigma\circ\tau(i),k}\delta_{j,l},
\end{align*}
we obtain (iii).
\end{proof}
Assume $I$ is a Hopf ideal of $D(H)$ and $a\in D(H)$. For simple, we write $a$ as $a+I$ in $D(H)/I$.
\begin{lemma}\label{lem3.1.x}
Assume $H$ satisfies the conditions of Theorem \ref{thm1.x}. Then
\begin{itemize}
  \item[(i)] $\chi x\in Z(D(H))$, where $Z(D(H))$ is the center of $D(H)$.
  \item[(ii)]  $\{a_iE_{a_jx^k}|\;1\leq i,j\leq m,\;1\leq k\leq n\}$ is linear basis of $D(H)/I$,
  \item[(iii)] $\{E_{a_ix^j}a_k|\;1\leq i,k\leq m,\;1\leq j\leq n\}$ is linear basis of $D(H)/I$.
\end{itemize}

\end{lemma}

\begin{proof}
To prove (i), we only need to show $a_i(\chi x)=(\chi x)a_i$ and $E_{a_ix^j}(\chi x)=(\chi x)E_{a_ix^j}$ for $1\leq i\leq m,\;1\leq j\leq n$. Directly, we have
\begin{align*}
a_i(\chi x)&=\chi(\chi^{-1}\rightharpoonup a_i\leftharpoonup \chi)x\\
                &=\chi(a_{\sigma\circ \tau(i)}\leftharpoonup \chi)x\\
                &=\chi a_{\sigma(i)}x=(\chi x)a_i.
\end{align*}
Since (ii)-(iii) of Lemma \ref{lem3.0.x}, we get $E_{a_ix^j}(\chi x)=\chi(x)^jE_{a_{\sigma\circ \tau (i)}x^j}x$ and $(\chi x)E_{a_ix^j}=\chi(x)^jxE_{a_{\tau (i)}x^j}$. Next, we show $E_{a_{\sigma\circ \tau (i)}x^j}x=xE_{a_{\tau (i)}x^j}$. By definition, we know $xE_{a_{\tau (i)}x^j}=(x\rightharpoonup E_{a_{\tau (i)}x^j}\leftharpoonup x^{-1} )x$. Note that $x^{-1}a_k=a_{\sigma^{-1}(k)}x^{-1}$ and so $(x\rightharpoonup E_{a_{\tau (i)}x^j}\leftharpoonup x^{-1} )=E_{a_{\sigma\circ \tau (i)}x^j}$. This implies $E_{a_{\sigma\circ \tau (i)}x^j}x=xE_{a_{\tau (i)}x^j}$.

Since we have shown $\chi x\in Z(D(H))$, we know $\langle \chi x\rangle\xrightarrow{\iota} D(H) \xrightarrow{\pi} D(H)/I$ is exact sequence. Thus $\dim(D(H)/I)=\frac{\dim(D(H))}{n}$. Denote $V$ as the linear subspace which is spanned by $\{a_iE_{a_jx^k}|\;1\leq i,j\leq m,\;1\leq k\leq n\}$. Then we have $\dim(V)\leq \frac{\dim(D(H))}{n}$. Note that $a_i x^jE_{a_kx^l}=\chi(x)^{-jl}a_iE_{a_kx^l}$ in $D(H)/I$ by Lemma \ref{lem3.0.x}, hence $V=D(H)/I$. Due to dimension reason, we know (ii). Similarly, we have (iii).
\end{proof}

\textbf{Proof of Theorem \ref{thm1.x}.} Due to Lemma \ref{lem3.1.x}, we know $\langle \chi x\rangle$ is normal subHopf algebra. To complete the proof, we will show $g_{\overline{\mathcal{R}_{21}}\overline{\mathcal{R}}}$ is surjective. By definition, we know
\begin{align*}
\mathcal{R}_{21}\mathcal{R}&=\sum\limits_{1\leq i,k\leq m}\sum\limits_{1\leq j,l\leq n}E_{a_kx^l}a_ix^{j}\otimes a_k x^l E_{a_ix^j}.
\end{align*}
Denote $\omega$ as $\chi(x)$. Since $x=\chi^{-1}$ in $D(H)/I$ and $\chi^{-l} E_{a_ix^j}=w^{-jl}E_{a_{\tau^{-l}(i)}x^j}$ by Lemma \ref{lem3.0.x}, we get
\begin{align*}
\overline{\mathcal{R}_{21}\mathcal{R}}&=\sum\limits_{1\leq i,k\leq m}\sum\limits_{1\leq j,l\leq n}E_{a_kx^l}a_ix^{j}\otimes \omega^{-jl} a_k E_{a_{\tau^{-l}(i)}x^j}\\
&=\sum\limits_{1\leq i,k\leq m}\sum\limits_{1\leq j,l\leq n}(E _{a_kx^{-l}})(a_{\tau^{-l}(i)})x^{j}\otimes \omega^{jl} a_k E_{a_{i}x^j}.
\end{align*}
By (ii) of Lemma \ref{lem3.1.x}, $\{a_iE_{a_jx^k}|\;1\leq i,j\leq m,\;1\leq k\leq n\}$ is linear basis of $D(H)/I$. Denote its dual basis as $\{A(i,j,k)|\;1\leq i,j\leq m,\;1\leq k\leq n\}$. By definition,  $g_{\overline{\mathcal{R}_{21}}\overline{\mathcal{R}}}(A(k,i,j))=\sum_{1\leq l\leq n}\omega^{jl}(E _{a_kx^{-l}})(a_{\tau^{-l}(i)})x^{j}$. Since $xa_i=a_{\sigma(i)}x$ for $1\leq i\leq m$, we have $a_{\tau^{-l}(i)}x^{j}=x^j a_{\sigma^{-j}\circ \tau^{-l}(i)}$. Hence
\begin{align*}
g_{\overline{\mathcal{R}_{21}}\overline{\mathcal{R}}}(A(k,i,j))&=\sum_{1\leq l\leq n}\omega^{jl}(E _{a_kx^{-l}})x^j (a_{\sigma^{-j}\circ \tau^{-l}(i)})\\
&=\sum_{1\leq l\leq n}\omega^{jl}(E _{a_kx^{-l}})\chi^{-j} (a_{\sigma^{-j}\circ \tau^{-l}(i)})\\
&=\sum_{1\leq l\leq n}\omega^{jl}[\omega^{jl}E _{a_{(\tau\circ \sigma)^{-j}(k)}x^{-l}}] (a_{\sigma^{-j}\circ \tau^{-l}(i)})\quad (\text{Lemma } \ref{lem3.0.x})\\
&=\sum_{1\leq l\leq n}\omega^{2jl}[E _{a_{(\tau\circ \sigma)^{-j}(k)}x^{-l}}] (a_{\sigma^{-j}\circ \tau^{-l}(i)}).
\end{align*}
Note that $\sigma\tau=\tau \sigma$ by Lemma \ref{lem3.0.x}, we know
\begin{align*}
g_{\overline{\mathcal{R}_{21}}\overline{\mathcal{R}}}(A((\tau \circ \sigma)^j(k),\sigma^{j}(i),j))=\sum_{1\leq l\leq n}\omega^{2jl}[E _{a_{k}x^{-l}}] (a_{\tau^{-l}(i)}).
\end{align*}
Since $n$ is odd, we know $\omega^2$ is primitive $n$-th root of unity. This implies
\begin{align*}
\sum_{1\leq j\leq n}\omega^{-2jl'}g_{\overline{\mathcal{R}_{21}}\overline{\mathcal{R}}}(A((\tau \circ \sigma)^j(k),\sigma^{j}(i),j))=n[E _{a_{k}x^{-l'}}] (a_{\tau^{-l'}(i)}),
\end{align*}
where $1\leq l'\leq n$. Due to Lemma \ref{lem3.1.x} and the above $i\in \{1,...,m\}$ is arbitrary, we know $g_{\overline{\mathcal{R}_{21}}\overline{\mathcal{R}}}$ is surjective.
\qed

In order to use Theorem \ref{thm1.x} more flexibly, we give the following result.
\begin{corollary} \label{coro5.x}
Assume $n$ is odd number and $x\in G(H)$ with order $n$. If $\{a_i|\;1\leq i\leq m\}\subseteq H$ and $\chi\in G(H^*),\;\sigma,\tau\in S_m$ such that
\begin{itemize}
 \item[(i)] $xa_i=a_{\sigma(i)}x$ and $\{a_ix^j|\;1\leq i\leq m,\;1\leq j\leq n\}$ is linear basis of $H$;
  \item[(ii)] $\chi E_{a_ix^j}=\chi(x)^jE_{a_{\tau(i)}x^j}$ and $E_{a_ix^j}\chi=\chi(x)^jE_{a_{\sigma\circ\tau(i)}x^j}$,
  \item[(iii)] $\chi(x)$ is primitive $n$th root of unity and $|\chi|$=$n$;
\end{itemize}
then $\langle \chi x\rangle\subseteq D(H)$ is normal subHopf algebra and $(D(H)/I,\overline{\mathcal{R}})$ is factorizable Hopf algebra, where $I=D(H)\langle x\rangle^{+}$ and $\mathcal{R}$ is standard universal $\mathcal{R}$-matrix of $D(H)$.
\end{corollary}

\begin{proof}
Note that the condition (ii) is equivalent to the condition (ii) of Theorem \ref{thm1.x} due to the proof of Lemma \ref{lem3.0.x}.
\end{proof}

\section{Some applications}
In this section, we first apply Theorem \ref{thm1.x} to construct two families of ribbon factorizable Hopf algebras and then determine the fusion rings of unitary modular tensor categories which are arising from one of them.
\subsection{Two families of factorizable Hopf algebras.} Recall that we have defined Hopf algebras $A(G,\sigma,n)$ in Example \ref{ex2.1.x}. Let $\chi:A(G,\sigma,n)\rightarrow \Bbbk$ be the algebra map which is determined by $\chi(e_g)=\delta_{g,b}$ and $\chi(x)=\omega$, where $\omega$ is primitive $n$-th root of unity. Then we have
\begin{proposition}\label{pro3.1.x}
Suppose $n$ is odd. Then the Hopf algebra $D(A(G,\sigma,n))/I$ is factorizable Hopf algebra with dimension $n|G|^2$, where $I=D(A(G,\sigma,n))\langle \chi x\rangle^{+}$.
\end{proposition}

\begin{proof}
We only need to prove $A(G,\sigma,n)$ satisfies the conditions of Theorem \ref{thm1.x}. Since $\{e_gx^j|\;g\in G, 1\leq j\leq q\}$ is linear basis of $A(G,\sigma,n)$ and $xe_g=e_{g\triangleleft x^{-1}}x$, the condition (i) of Theorem \ref{thm1.x} holds. Define $\sigma,\tau:G\longrightarrow G$ by $\sigma(g)=g\triangleleft x^{-1}$ and $\tau(g)=bg$ for $g\in G$. Then it can be seen that $e_g\leftharpoonup \chi=e_{\tau^{-1}(g)}$ and $\chi\rightharpoonup e_g=e_{\tau^{-1}\circ \sigma^{-1}(g)}$ for $g\in G$, hence the condition (ii) of Theorem \ref{thm1.x} holds. Directly, we have $|\chi|=n$ and thus (iii) of Theorem \ref{thm1.x} holds.
\end{proof}
Since $\mathscr{A}_l$ belongs to $A(G,\sigma,n)$ and $\mathscr{A}_l$ over $\mathbb{C}$ is $C^*$-Hopf algebra, we have
\begin{corollary}\label{cor4.2.x}
The Hopf algebra $D(\mathscr{A}_l)/I$ over $\mathbb{C}$ is factorizable $C^*$-Hopf algebra with dimension $p^2q^3$, where $I=D(\mathscr{A}_l)\langle \chi x\rangle^{+}$.
\end{corollary}

\begin{proof}
Since $I$ is Hopf $*$-ideal and Proposition \ref{pro3.1.x}, we get what we want.
\end{proof}

Another family of ribbon factorizable Hopf algebra is constructed by using the $n$-rank Taft algebra $\overline{\mathscr{A}}_q(n)$. Let $\chi:\overline{\mathscr{A}}_q(n)\longrightarrow \Bbbk$ be the algebra map which is determined by $\chi(x_i)=0,\;\chi(g_i)=q$, where $1\leq i \leq n$. To use Theorem \ref{thm1.x}, we define $e_{(i_1,...,i_n),(j_1,...,j_{n-1})}\in \overline{\mathscr{A}}_q(n)$ as follows:
\begin{align*}
e_{(i_1,...i_n),(j_1,...j_{n-1})}&=\frac{1}{l^{2n-1}}\sum_{\substack{0\leq k_1,\dots,k_{n}\leq l-1\\   0\leq t_1,\dots,t_{n-1}\leq l-1}}(\prod\limits_{u=1}^{n}q^{i_uk_u})(\prod\limits_{v=1}^{n-1}q^{j_vt_v})(\prod\limits_{w=1}^{n}x_w^{k_w})(\prod\limits_{s=1}^{n-1}g_s^{t_s}),
\end{align*}
where $0\leq i_1,\dots,i_{n}\leq l-1$ and $0\leq j_1,\dots,j_{n-1}\leq l-1$. Since $q$ is primitive $l$-th root of unity, we have
\begin{align*}
(\prod\limits_{w=1}^{n}x_w^{k_w})(\prod\limits_{s=1}^{n-1}g_s^{t_s})&=\sum_{\substack{0\leq i_1,\dots,i_{n}\leq l-1\\   0\leq j_1,\dots,j_{n-1}\leq l-1}}(\prod\limits_{u=1}^{n}q^{-i_uk_u})(\prod\limits_{v=1}^{n-1}q^{-j_vt_v})e_{(i_1,...i_n),(j_1,...j_{n-1})}.
\end{align*}
Thus $\{(e_{(i_1,...i_n),(j_1,...j_{n-1})})g_n^{j_n}|\;0\leq i_1,\dots,i_{n}\leq l-1$ and $0\leq j_1,\dots,j_{n}\leq l-1\}$ is linear basis of $\overline{\mathscr{A}}_q(n)$. Define $M:=\{((i_1,...i_n),(j_1,...j_{n-1}))|\;0\leq i_1,\dots,i_{n}\leq l-1,\;0\leq j_1,\dots,j_{n-1}\leq l-1\}$.

\begin{proposition}\label{pro3.2.x}
If $l$ is odd, then the Hopf algebra $D(\overline{\mathscr{A}}_q(n))/I$ is factorizable Hopf algebra with dimension $l^{4n-1}$, where $I=D(\overline{\mathscr{A}}_q(n))\langle \chi g_n\rangle^{+}$.
\end{proposition}

\begin{proof}
We only need to prove $\overline{\mathscr{A}}_q(n)$ satisfies the conditions of Theorem \ref{thm1.x}. Define $\sigma,\tau:M\longrightarrow M$ by $\sigma(((i_1,...i_n),(j_1,...j_{n-1})))=((i_1+1,...i_n+1),(j_1,...j_{n-1}))$ and $\tau(((i_1,...i_n),(j_1,...j_{n-1})))=((i_1-1,...i_n-1),(j_1-1,...j_{n-1}-1))$ for $0\leq i_1,\dots,i_{n}\leq l-1,\;0\leq j_1,\dots,j_{n-1}\leq l-1$, here we agree that $-1=l-1$ and $l=0$. Then it can be seen that $g_n (e_{(i_1,...i_n),(j_1,...j_{n-1})})=e_{\sigma((i_1,...i_n),(j_1,...j_{n-1}))}$. Moreover, we have
\begin{align*}
(e_{(i_1,...i_n),(j_1,...j_{n-1})})\leftharpoonup \chi=e_{\tau^{-1}((i_1,...i_n),(j_1,...j_{n-1}))}
\end{align*}
and
\begin{align*}
\chi\rightharpoonup (e_{(i_1,...i_n),(j_1,...j_{n-1})})=e_{\tau^{-1}\circ \sigma^{-1}((i_1,...i_n),(j_1,...j_{n-1}))},
\end{align*}
where $0\leq i_1,\dots,i_{n}\leq l-1,\;0\leq j_1,\dots,j_{n-1}\leq l-1$. Hence the condition (ii) of Theorem \ref{thm1.x} holds. By definition, we have $|\chi|=l$ and $\chi(g_n)$ is primitive $l$-th root of unity. Thus (iii) of Theorem \ref{thm1.x} holds.
\end{proof}
\begin{remark}\label{rk1}
\emph{If $n=1$ then $D(\overline{\mathscr{A}}_q(n))/I$ is exactly the small quantum group $u_q(sl_2)$. If we adjust the parameter $l,n$, then $D(\overline{\mathscr{A}}_q(n))/I$ will be different from Drinfel'd double and small quantum groups due to the dimension reason. Hence we obtain large number of new pointed factorizable Hopf algebras $D(\overline{\mathscr{A}}_q(n))/I$.}
\end{remark}
To construct MTCs, we introduce the following result.
\begin{theorem}\label{thm4.x}
If $l$ is odd, then the representation category of $D(\overline{\mathscr{A}}_q(n))/I$ is non-semisimple modular tensor category.
\end{theorem}

\begin{proof}
Denote $H$ as $\overline{\mathscr{A}}_q(n)$. By Proposition \ref{pro3.2.x}, we only need to show $(D(H),\mathcal{R})$ is ribbon Hopf algebra. Let $\Lambda=(\sum_{s\in G(H)}s)(x_1^{l-1}\dots x_n^{l-1})$. By definition $\Lambda$ is left integral of $H$. This implies $\alpha$ (the distinguished grouplike element of $H^*$) is determined by $\alpha(g_i)=q^{-n+2i}$ and $\alpha(x_i)=0$ for $1\leq i \leq n$. Similarly, one can get $g=g_1^{-1}\dots g_n^{-1}$ (the distinguished grouplike element of $H$). Assume $l=2m-1$. Let $a=g^m$ and $\beta=\alpha^m$. By definition, we know that $a, \beta$ satisfy the conditions of Theorem \ref{thm0.x}. Thus $(D(H),\mathcal{R})$ is ribbon Hopf algebra.
\end{proof}

\subsection{A family of unitary modular tensor categories.}
In this subsection, we will determine the fusion rings of unitary modular tensor categories which are arising from a family of $C^*$-Hopf algebras. We assume $\Bbbk=\mathbb{C}$ in this subsection. Recall that $C^*$-Hopf algebras $\mathscr{A}_l$ are given by Example \ref{ex2.1.5}. Since Corollary \ref{cor4.2.x}, we know that the Hopf algebras $D(\mathscr{A}_l)/I$ are factorizable $C^*$-Hopf algebra with dimension $p^2q^3$, where $I=D(\mathscr{A}_l)\langle \chi x\rangle^{+}$. These Hopf algebras are not Drinfel'd double since the dimension reason. Hence a large number of new unitary modular tensor categories are gotten. Next, we will only consider the case $D(\mathscr{A}_0)/I$ since the other cases can be done in a similar way. Denote $A_{p,q}$ as $D(\mathscr{A}_0)/I$ for simplicity. Next we study the unitary modular tensor category arising from $A_{p,q}$. To do this, we first describe $A_{p,q}$ as follows:

\begin{theorem}\label{thm2.x}
The $C^*$-Hopf algebra $A_{p,q}$ is generated by $\{x,y,z_i,e_g|\;1\leq i \leq q,\;g\in G\}$ as algebra, with the relations
$$
yx=xy^t,\;z_ix=xz_i,\;e_gx=xe_{g\triangleleft x},\;z_iy=yz_i,$$

$$
 e_gy=y\sum_{1\leq i\leq q} z_i (e_{a^{-1}\triangleleft x^i g a}),\;e_gz_i=z_ie_g,\;x^q=y^p=1,\;z_iz_j=\delta_{i,j}z_i,\;e_ge_h=\delta_{g,h}e_g.$$
The coproduct, counit, involution and antipode are given by
$$
\Delta(x) = x\otimes x, \;\Delta(y) =\sum_{1\leq i\leq q} y^{t^i}\otimes yz_i,
$$

$$
\Delta(z_i) =\sum_{1\leq j\leq q} z_j \otimes z_{i-j},\;\Delta(e_g) =\sum_{h\in G} e_h \otimes e_{h^{-1}g},
$$

$$
\epsilon(x)=\epsilon(y)=1, \;\epsilon(z_i)=\delta_{i,0},\;\epsilon(e_g)=\delta_{g,1},
$$

$$
x^* =x^{-1},\; y^* =y^{-1}, \;z_i^{*} =z_{i},\;e_g^* =e_{g}.
$$

$$
S(x) =x^{-1},\; S(y) =\sum_{1\leq i\leq q}y^{-t^{-i}}z_i, \;S(z_i) =z_{-i},\;S(e_g) =e_{g^{-1}}.
$$
\end{theorem}
To show above theorem, we introduce some lemmas as follows. Assume $\{E_{g;x^i}|\;g\in G, 1\leq i\leq q\}$ is the dual basis of $\{e_{g}x^i|\;g\in G, 1\leq i\leq q\}$ for $\mathscr{A}_0^*$, i.e. $E_{g;x^i}(e_h x^j)=\delta_{g,h}\delta_{i,j}$ for $g,h\in G$ and $1\leq i,j\leq q$. Let $\omega$ be a primitive $q$-th root of unity. To find algebraic generator of $\mathscr{A}_0^*$, we define $Y,\;Z_i(1\leq i\leq q),\;\chi$ as follows: $$Y=\sum_{1\leq j\leq q}E_{a;x^j},\;Z_i=E_{1,x^i},\;\chi=\sum_{1\leq j\leq q}\omega^jE_{b;x^j}.$$

\begin{lemma}\label{lem3.2.x}
The following statements hold for $\mathscr{A}_0^*$
\begin{itemize}
  \item[(i)] $(E_{g;x^i})(E_{h;x^j})=\delta_{i,j}E_{gh;x^i}$,
  \item[(ii)]  $\Delta(E_{g;x^i})=\sum_{1\leq j\leq q} E_{g;x^j}\otimes E_{g\triangleleft x^j;x^{i-j}}$,
  \item[(iii)]  $(E_{g;x^i})^\ast=E_{g^{-1};x^i}$, where $\ast$ is the involution.
  \item[(iv)] $E_{a^ib^j;x^k}=\omega^{-jk}Y^i \chi^j Z_k$.
\end{itemize}
\end{lemma}

\begin{proof}
By definition of $\mathscr{A}_0$, we have $\Delta(e_kx^l)=\sum_{u\in G}e_ux^l\otimes e_{u^{-1}k}x^l$. This implies $(E_{g;x^i}E_{h;x^j})(e_kx^l)=\delta_{i,j}\delta_{gh,k}\delta_{i,l}$. Hence we have (i).

Since $(e_hx^j)(e_kx^l)=(\delta_{h,k\triangleleft x^{-j}}e_h x^{j+l})$, we get
$E_{g;x^i}[(e_hx^j)(e_kx^l)]=\delta_{h,k\triangleleft x^{-j}}\delta_{g,h}\delta_{i,j+l}$. This implies (ii).

By definition, we have $E_{g;x^i}^{*}(e_hx^j)=\overline{E_{g;x^i}[S(e_hx^j)^*]}$. Since $S(e_hx^j)^*=e_{h^{-1}}x^j$, we get (iii).

Due to (i), we know $Y^i \chi^j Z_k=\omega^{jk}E_{a^ib^j;x^k}$. And so (iv) holds.
\end{proof}
As a corollary of Lemma \ref{lem3.2.x}, we have
\begin{corollary}\label{lem3.3.x}
The $C^*$-Hopf algebra $\mathscr{A}_0^*$ is generated by $\{Y,Z_i,\chi|\;1\leq i\leq q\}$ as algebra, with relations
$$
Z_iY=YZ_i,\;\chi Y=Y^t \chi,\;\chi Z_i=Z_i\chi,\;Y^p=\chi^q=1,\;Z_iZ_j=\delta_{i,j}Z_i,$$
The coproduct, counit, involution and antipode are given by
$$
\Delta(Y) = \sum_{1\leq i\leq q} YZ_i\otimes Y^{t^i}, \;\Delta(Z_i) =\sum_{1\leq j\leq q} Z_j\otimes Z_{i-j},\;\Delta(\chi) =\chi\otimes \chi,
$$

$$
\epsilon(Y)=\epsilon(\chi)=1, \;\epsilon(Z_i)=\delta_{i,0},
$$

$$
Y^* =Y^{-1}, \;Z_i^{*} =Z_{i},\;\chi^* =\chi^{-1},
$$

$$
S(Y) =\sum_{1\leq i\leq q}Y^{-t^{-i}}Z_i, \;S(Z_i) =Z_{-i},\;S(\chi) =\chi^{-1}.
$$
\end{corollary}

\begin{proof}
By Lemma \ref{lem3.2.x}, the only non-trivial things are the following equalities:
$$\chi Y=Y^t \chi,\;\Delta(Y) = \sum_{1\leq i\leq q} YZ_i\otimes Y^{t^i},\;S(Y) =\sum_{1\leq i\leq q}Y^{-t^{-i}}Z_i.$$

Since Lemma \ref{lem3.2.x}, we have $\chi Y=\sum_{1\leq i\leq q}\omega^i E_{ba;x^i}$ and $Y^t\chi=\sum_{1\leq i\leq q}\omega^i E_{a^tb;x^i}$. By definition of $G$, $ba=a^tb$ and hence $\chi Y=Y^t \chi$.

Using Lemma \ref{lem3.2.x} again, we have $\Delta(Y) = \sum_{1\leq i,j\leq q} E_{a;x^j}\otimes E_{a^{t^j};x^{i-j}}$. Note that $E_{a;x^j}=YZ_j$ and $E_{a^{t^j};x^{i-j}}=Y^{t^j}Z_{i-j}$ by Lemma \ref{lem3.2.x}, thus $\Delta(Y) = \sum_{1\leq i,j\leq q} YZ_j\otimes Y^{t^j}Z_{i-j}$. Since $\sum_{1\leq j\leq q}Z_{i-j}=1$, we know $\Delta(Y) = \sum_{1\leq i\leq q} YZ_i\otimes Y^{t^i}$.

By definition, $E_{a;x^i}(S(e_g x^j))=E_{a;x^i}(e_{g\triangleleft x^j} x^{-j})=\delta_{g,a\triangleleft x^i}\delta_{j,-i}$. Hence $S(E_{a;x^i})=E_{a\triangleleft x^i;x^{-i}}$. Note that $E_{a\triangleleft x^i;x^{-i}}=Y^{t^i}Z_{-i}$, thus $S(Y) =\sum_{1\leq i\leq q}Y^{-t^{-i}}Z_i$.
\end{proof}
Next, we show Theorem \ref{thm2.x}. For simple, we write $a$ as $a+I$ in $A_{p,q}$.

\textbf{Proof of Theorem \ref{thm2.x}.}
Denote $y=Y$ and $z_i=Z_i$. By definition of $A_{p,q}$, we have $\chi=x^{-1}$. Combine this with Corollary \ref{lem3.3.x}, we know $A_{p,q}$ is generated by $\{x,y,z_i,e_g|\;1\leq i \leq q,\;g\in G\}$ as algebra. Since Corollary \ref{lem3.3.x} and the definition of $\mathscr{A}_0$, we only need to prove the following equalities:
$$yx=xy^t,\;z_ix=xz_i,\;e_gy=y\sum_{1\leq i\leq q} z_i (e_{a^{-1}\triangleleft x^i g a}),\;e_gz_i=z_ie_g.$$

By Corollary \ref{lem3.3.x}, we have $\chi Y=Y^t \chi$ and $Z_i \chi= \chi Z_i$. Since $\chi=x^{-1}$ in $A_{p,q}$, we get $yx=xy^t,\;z_ix=xz_i$. By definition of $A_{p,q}$, we get
$$e_g Y=\sum_{1\leq i\leq q}\sum_{hkl=g}(e_h\rightharpoonup E_{a;x^i}\leftharpoonup e_{l^{-1}})e_k.$$
Since $(e_h\rightharpoonup E_{a;x^i}\leftharpoonup e_{l^{-1}})(e_s x^r)=\delta_{s,a}\delta_{r,i}\delta_{l,a^{-1}}\delta_{h,a\triangleleft x^i}$, we get $(e_h\rightharpoonup E_{a;x^i}\leftharpoonup e_{l^{-1}})=\delta_{l,a^{-1}}\delta_{h,a\triangleleft x^i}E_{a;x^i}$. Note that $E_{a;x^i}=YZ_i$, hence $e_gy=y\sum_{1\leq i\leq q} z_i (e_{a^{-1}\triangleleft x^i g a})$. Similarly, we have $e_gZ_i=Z_ie_g$. And so $e_gz_i=z_ie_g$.
\qed

Next, we determine the fusion ring of $Rep(A_{p,q})$ which is the $C^*$-tensor category of representations of $A_{p,q}$ on finite dimensional Hilbert spaces. Define $B=\langle e_g,z_i|\;g\in G,1\leq i\leq q\rangle$ as algebra. Note that $B$ is actually diagonal algebra, hence the representation category of $B$ is very simple. To determine representations of $A_{p,q}$, our idea is to use the representation category of $B$ and add some additional conditions.

Assume $\omega$(resp. $\eta$) is primitive $q$th(resp. $p$th) root of unity in the following content. Since the definition of $\mathscr{A}_0$, we can assume $p-1=mq$ and $\mathbb{Z}_p^\times=\langle \beta \rangle$ as group, where $\mathbb{Z}_p^\times$ is the multiplicative $\mathbb{Z}_p-\{0\}$. Recall that $t^q=1$, so we can assume $t=\beta^m$. And this implies $\mathbb{Z}_p^\times=\cup_{i=1}^m \beta^i \langle t\rangle$, where $\langle t\rangle$ is the multiplicative subgroup of $\mathbb{Z}_p^\times$. Then we construct four series of simple representations on Hilbert spaces as follows:
\begin{enumerate}
\item[(1)] $T_{i,j}=\langle u \rangle,\;1\leq i,j\leq q$ as Hilbert space: the action of $A_{p,q}$ is determined by
\begin{align*}
x.u=\omega^j u,\;y.u=u,
\end{align*}
\begin{align*}
z_k.u=\delta_{k,i}u,\;e_{g}.u= \delta_{g,b^i}u,\;0\leq k\leq q-1;
\end{align*}
\item[(2)] $U_{i',j'}=\langle u_k|\;0\leq k\leq q-1 \rangle$ as Hilbert space $,1\leq i'\leq q,\;j'\in \{\beta^1, \dots , \beta^{m}\}$: the action of $A_{p,q}$ is determined by
\begin{align*}
x.u_k=u_{k+1},\;y.u_k=\eta^{j't^k}u_k,
\end{align*}
\begin{align*}
z_l.u_{k}=\delta_{l,i'}u_k,\;e_{g}.u_k= \delta_{g,b^{i'}}u_k,\;0\leq k,l\leq q-1;
\end{align*}
\item[(3)] $V_{(i'',j'',k')}=\langle u_l|\;0\leq l\leq q-1 \rangle$ as Hilbert space, where $i''\in \{\beta^1, \dots , \beta^{m}\},\;1\leq j''\leq q,\;1\leq k'\leq p$: the action of $A_{p,q}$ is determined by
\begin{align*}
x.u_l=u_{l-1},\;y.u_l=\eta^{k't^{-l}}u_l,
\end{align*}
\begin{align*}
z_s.u_{l}=\delta_{s,j''}u_l,\;e_{g}.u_l= \delta_{g,a^{i''t^l}b^{j''}}u_l,\;0\leq l,s\leq q-1;
\end{align*}

\item[(4)] $W_{(i''',j''',k'')}=\langle v_l|\;0\leq l\leq p-1 \rangle$ as Hilbert space, where $1\leq i'''\neq j'''\leq q,\;1\leq k''\leq q$: the action of $A_{p,q}$ is determined by
\begin{align*}
x.v_l=\omega^{k''} v_{t^{-1}l},\;y.v_l=v_{l+1},\;0\leq l\leq p-1
\end{align*}
\begin{align*}
z_s.v_{l}=\delta_{s,j'''}v_l,\;e_{g}.v_l= \delta_{g,a^{l(t^{j'''}-t^{i'''})}b^{i'''}}v_l,\;0\leq s\leq q-1;
\end{align*}
\end{enumerate}
The inner products on above Hilbert spaces are assumed to be standard, for example $W_{(i''',j''',k'')}=\langle v_l|\;0\leq l\leq p-1 \rangle$ and $\langle v_i,v_j\rangle=\delta_{i,j}$ for $0\leq i,j\leq p-1$. To collect above modules, we define $M=\{T_{i,j},\;U_{i',j'},\;V_{(i'',j'',k')},\;W_{(i''',j''',k'')}|\;1\leq i,j,i',j'',k''\leq q,\;1\leq i'''\neq j'''\leq q,\;j',i''\in \{\beta^1, \dots , \beta^{m}\}\}$. Denote the fusion ring of $Rep(A_{p,q})$ as $K(p,q)$. Then we have
\begin{theorem}\label{thm3.x}
The rank of $K(p,q)$ is $q^2(p^2+q-1)$ and $M$ is a basis of $K(p,q)$. Moreover, the dimensional function of $K(p,q)$ is given as follows:
$$\dim(T_{i,j})=1,\;\dim(U_{i',j'})=\dim(V_{(i'',j'',k')})=q,\;\dim(W_{(i''',j''',k'')})=p,$$
and the unity is $T_{0,0}$, the multiplication of $K(p,q)$ is given as follows:
\begin{align*}
&T_{i,j}.T_{i',j'}=T_{i+i',j+j'},\;T_{i,j}.U_{i'',j''}=U_{i+i'',j''},\\
&T_{i,j}.V_{(k,l,s)}=V_{(k,i+l,t^is)},\;T_{i,j}.W_{(k',l',s')}=W_{(i+k',i+l',j+s')},\\
&U_{i'',j''}.U_{i''',j'''}=\sum_{s=0}^{q-1}U(i''+i''',\alpha_s),\; \text{if } j'''\notin -j''\langle t\rangle, \text{here}\\
&j''t^{i'''}+j'''t^s=\alpha_s t^{\beta_s} \text{for }
0\leq s\leq q-1 \;\text{and}\;\alpha_s \in \{\beta^1, \dots , \beta^{m}\},\\
&U_{i'',j''}.U_{i''',j'''}=\sum_{s=0,s\neq i_0}^{q-1}U(i''+i''',\alpha_s)+\sum_{s=0}^{q-1}T(i''+i''',s),\;\text{where }\\
&j''t^{i'''}+j'''t^{i_0}=0 \;\text{for some } 0\leq i_0\leq q-1,\\
&U_{i'',j''}.V_{(i''',j''',k'')}=\sum_{s=0}^{q-1}V_{(i''',i''+j''',t^{i''}(k''+j''t^{j'''+s}))},\\
&U_{i'',j''}.W_{(i''',j''',k'')}=\sum_{s=0}^{q-1}W_{(i''+i''',i''+j''',s)},\\
&V_{(i'',j'',k')}.V_{(i''',j''',k'')}=\sum_{s=0}^{q-1}V_{(\alpha_s',j''+j''',t^{\beta_s'}(k't^{j'''}+k''t^{-s}))}
\text{ if } i''\notin -i'''\langle t\rangle\;\text{and } \\
&i''+i'''t^{s+j''}=\alpha_s' t^{\beta_s'} \;\text{where } \alpha_s'\in \{\beta^1\dots \beta^m\}.
\end{align*}
\begin{align*}
&V_{(i'',j'',k')}.V_{(i''',j''',k'')}=\sum_{s=0,s\neq s_0}^{q-1}V_{(\alpha_s',j''+j''',t^{\beta_s'}(k't^{j'''}+k''t^{-s}))}+\sum_{s=0}^{q-1}T_{j''+j''',s},
\text{ if} \\
&i''+i'''t^{s_0+j''}=0 \text{ for some } 0\leq s_0\leq q-1 \text{ and } k't^{j'''}+k''t^{-s_0}=0,\\
&V_{(i'',j'',k')}.V_{(i''',j''',k'')}=\sum_{s=0,s\neq s_0}^{q-1}V_{(\alpha_s',j''+j''',t^{\beta_s'}(k't^{j'''}+k''t^{-s}))}+U_{j''+j''',\alpha_{s_0}''}\text{ if }\\
&i''+i'''t^{s_0+j''}=0 \text{ for some } 0\leq s_0\leq q-1 \text{ and }k't^{j'''}+k''t^{-s_0}=\alpha_{s_0}'' t^{\beta_{s_0}''},\\
&\text{where } \alpha_s''\in \{\beta^1\dots \beta^m\}.
\end{align*}
\begin{align*}
&V_{(i'',j'',k')}.W_{(i''',j''',k'')}=\sum_{s=0}^{q-1}W_{(j''+i''',j''+j''',s)},\\
&W_{(i,j,k)}.W_{(i',j',k')}=\sum_{s=0}^{q-1}mW_{(i+i',j+j',s)}+W_{(i+i',j+j',k+k')} \text{ if } i+i'\neq j+j',\\
&W_{(i,j,k)}.W_{(i',j',k')}=T_{i+i',k+k'}+\sum_{s=1}^m U_{i+i',\beta^s}+\sum_{l=0}^{p-1}\sum_{s=0}^{q-1}V_{(l,i+i',s)} \text{ if } i+i'=j+j'.
\end{align*}

\end{theorem}
Next, we introduce some lemmas to show above theorem.
\begin{lemma}\label{lem3.4.x}
The set $M$ is a basis of $K(p,q)$.
\end{lemma}

\begin{proof}
Firstly, we show they are simple objects. We choose $W_{(i''',j''',k'')}$ to check since the other case can be done in similar way. To show $W_{(i''',j''',k'')}$ is $\ast$-representation of $A_{p,q}$, we need to prove the algebra relations hold and it keeps the involution. The only non-trivial cases are as follows:
$$e_gx=xe_{g\triangleleft x},\;e_gy=y\sum_{1\leq i\leq q} z_i (e_{a^{-1}\triangleleft x^i g a}).$$
By definition, $(e_gx).v_l=\omega^{k''}(\delta_{g,{a^{lt^{-1}d}}b^{i'''}})v_{t^{-1}l}$ and
$(xe_{g\triangleleft x}).v_l=\omega^{k''}(\delta_{g\triangleleft x,{a^{ld}}b^{i'''}})v_{t^{-1}l}$, where $d=t^{j'''}-t^{i'''}$. Since $a\triangleleft x^{-1}=a^{t^{-1}}$ and $b\triangleleft x^{-1}=b$, we get $\delta_{g,{a^{lt^{-1}d}}b^{i'''}}=\delta_{g\triangleleft x,{a^{ld}}b^{i'''}}$. Hence $(e_gx).v_l=(xe_{g\triangleleft x}).v_l$. Similarly, we have $(e_gy).v_l=(\delta_{g,{a^{(l+1)d}}b^{i'''}})v_{l+1}$ and $[y\sum_{1\leq i\leq q} z_i (e_{a^{-1}\triangleleft x^i g a})].v_l=(\delta_{a^{-1}\triangleleft x^{j'''} g a,{a^{ld}}b^{i'''}})v_{l+1}$, where $d=t^{j'''}-t^{i'''}$. Due to $a\triangleleft x^{-1}=a^{t^{-1}}$ and $b\triangleleft x^{-1}=b$, we get $\delta_{g,{a^{(l+1)d}}b^{i'''}}=\delta_{a^{-1}\triangleleft x^{j'''} g a,{a^{ld}}b^{i'''}}$. This implies $(e_gy).v_l=[y\sum_{1\leq i\leq q} z_i (e_{a^{-1}\triangleleft x^i g a})].v_l$. Next, we show $W_{(i''',j''',k'')}$ is simple. Suppose $0\neq W_0\subseteq W_{(i''',j''',k'')}$. Then we can find $v=\sum_{j=0}^{p-1}\lambda_j v_j\in W_0$ and some $0\neq \lambda_{i_0}\in \Bbbk$. Let $g_{i_0}=a^{di_0}b^{'''}$, where $d=t^{j'''}-t^{i'''}$. By definition, we have $e_{g_{i_0}}.v=\lambda_{i_0} v_{i_0}$. Thus $v_{i_0}\in W_0$. Since $y^{j-{i_0}}.v_{i_0}=v_j$ for $0\leq j\leq p-1$, we know $v_j\in W_0$. Hence $W_0=W$. This implies $W$ is simple.

Then we show the elements of $M$ are non-isomorphic. Recall that the subalgebra $B$ which is defined by $B=\langle e_g,z_i|\;g\in G,1\leq i\leq q\rangle$ as algebra. Directly one can know that $T_{i,j},\;U_{i',j'},\;V_{(i'',j'',k')},\;W_{(i''',j''',k'')}$ are non-isomorphic $B$ modules, thus the elements of $M$ are non-isomorphic as $A_{p,q}$ modules.

Define $T=\sum_{a\in M}\dim(a)a$. Note that $\dim(T)=p^2q^3=\dim(A_{p,q})$. Thus $M$ is a basis of $K(p,q)$.
\end{proof}

\begin{lemma}\label{lem3.5.x}
The fusion rule of $K(p,q)$ in Theorem \ref{thm3.x} hold.
\end{lemma}

\begin{proof}
(1):\;$T_{i,j}.T_{i',j'}=T_{i+i',j+j'}$. Assume $T_{i,j}=\langle u \rangle$ and $T_{i',j'}=\langle u' \rangle$. By definition, we get
\begin{align*}
e_g.(u\otimes u')=\delta_{g,b^{i+i'}}(u\otimes u'), \; z_k.(u\otimes u')=\delta_{k,j+j'}(u\otimes u').
\end{align*}
Thus $T_{i,j}\otimes T_{i',j'}\cong T_{i+i',j+j'}$.

(2):\;$T_{i,j}.U_{i'',j''}=U_{i+i'',j''}$. Assume $T_{i,j}=\langle u \rangle$ and $U_{i'',j''}=\langle u_k'|\;0\leq k\leq q-1 \rangle$. Let $\lambda_k=\omega^{jk}$ and $u_k''=u\otimes \lambda_ku_k'$, where $0\leq k\leq q-1$. Then we have
\begin{align*}
x.u_k'=u_{k+1}', \; y.u_k'=\eta^{j''t^k}u_k',\\
z_l.u_k'=\delta_{l,i+i'}u_k',\; e_g.u_k'=\delta_{g,b^{i+i'}}u_k'.
\end{align*}
Thus $T_{i,j}\otimes U_{i'',j''}\cong U_{i+i'',j''}$.

(3):\;$T_{i,j}.V_{(i'',j'',k')}=V_{(i'',i+j'',k't^i)}$. Assume $T_{i,j}=\langle u \rangle$ and $V_{(i'',j'',k')}=\langle u_l|\;0\leq l\leq q-1 \rangle$. Let $\lambda_k=\omega^{-jk}$ and $u_l'=u\otimes \lambda_lu_{l-i}$, where $0\leq l\leq q-1$. Then we have
\begin{align*}
x.u_l'=u_{l-1}', \; y.u_k'=\eta^{(k't^i)t^{-l}}u_l',\\
z_k.u_l'=\delta_{k,i+j''}u_l',\; e_g.u_l'=\delta_{g,a^{i''t^l}}u_l'.
\end{align*}
Thus $T_{i,j}\otimes V_{(i'',j'',k')}\cong V_{(i'',i+j'',k't^i)}$.

(4):\;$T_{i,j}.W_{(i''',j''',k'')}=W_{(i+i''',i+j''',j+k'')}$. Assume $T_{i,j}=\langle u \rangle$ and $W_{(i''',j''',k'')}=\langle v_l|\;0\leq l\leq q-1 \rangle$. Let $v_l'=u\otimes v_{l}$, where $0\leq l\leq p-1$. Then we have
\begin{align*}
x.v_l'=\omega^{j+k''} v_{t^{-1}l}', \; y.v_l'=v_{l+1}',
\end{align*}
\begin{align*}
z_k.v_l'=\delta_{k,i+j'''}v_l',\;e_g.v_l'=\delta_{g,a^{l(t^{i+j'''}-t^{i+i'''})}}v_l'.
\end{align*}
Thus $T_{i,j}.W_{(i''',j''',k'')}\cong W_{(i+i''',i+j''',j+k'')}$.

(5.1):\;$U_{i'',j''}.U_{i''',j'''}=\sum_{s=0}^{q-1}U(i''+i''',\alpha_s)$, if $j'''\notin -j''\langle t\rangle$, where $j''t^{i'''}+j'''t^s=\alpha_s t^{\beta_s}$ for $0\leq s\leq q-1$ and $\alpha_s \in \{\beta^1, \dots , \beta^{m}\}$. Suppose $U_{i'',j''}=\langle u_k|\;0\leq k\leq q-1 \rangle$ and $U_{i''',j'''}=\langle u_k'|\;0\leq k\leq q-1 \rangle$. Let $u_l^s=u_{l-\beta_s}\otimes u_{s+l-\beta_s}'$ and $U_s=\langle u_l^s|\;0\leq l\leq q-1 \rangle$, where $0\leq s\leq q-1$. Then we have
\begin{align*}
x.u_l^s=u_{l+1}^s, \; y.u_l^s=\eta^{\alpha_s t^l}u_l^s,
\end{align*}
\begin{align*}
z_k.u_l^s=\delta_{k,i''+i'''}u_l^s,\;e_g.u_l^s=\delta_{g,b^{i''+i'''}}u_l^s.
\end{align*}
Thus $U_s\cong U_{i''+i''',\alpha_s}$, where $j't^{i''}+j''t^s=\alpha_s t^{\beta_s}$ and $\alpha_s \in \{\beta^1, \dots , \beta^{m}\}$. Since the definition of $U_{i'',j''} \otimes U_{i''',j'''}$, we know $U_{i'',j''}\otimes U_{i''',j'''}=\oplus_{s=0}^{q-1}U(i''+i''',\alpha_s)$.

(5.2):\;$U_{i'',j''}.U_{i''',j'''}=\sum_{s=0,s\neq i_0}^{q-1}U(i''+i''',\alpha_s)+\sum_{s=0}^{q-1}T(i''+i''',s)$, where $j''t^{i'''}+j'''t^{i_0}=0$ for some $0\leq i_0\leq q-1$. Suppose $U_{i'',j''}=\langle u_k|\;0\leq k\leq q-1 \rangle$ and $U_{i''',j'''}=\langle u_k'|\;0\leq k\leq q-1 \rangle$. Assume $s\neq i_0$ and  $0\leq s\leq q-1$. Let $u_l^s=u_{l-\beta_s}\otimes u_{s+l-\beta_s}'$ and $U_s=\langle u_l^s|\;0\leq l\leq q-1 \rangle$. Let $v_k=\sum_{r=0}^{q-1}\omega^{rk}(u_r\otimes u_{r+i_0}')$ for $0\leq k\leq q-1$. Denote $T=\langle v_r|\;0\leq r\leq q-1 \rangle$ as vector space. Assume $s\neq i_0$ and  $0\leq s\leq q-1$. Then we have
\begin{align*}
x.u_l^s=u_{l+1}^s, \; y.u_l^s=\eta^{\alpha_s t^l}u_l^s,
\end{align*}
\begin{align*}
z_k.u_l^s=\delta_{k,i''+i'''}u_l^s,\;e_g.u_l^s=\delta_{g,b^{i''+i'''}}u_l^s,
\end{align*}
and
\begin{align*}
x.v_k=\omega^{-k}v_k, \; y.v_k=v_k,
\end{align*}
\begin{align*}
z_s.v_k=\delta_{s,i''+i'''}v_k,\;e_g.v_k=\delta_{g,b^{i''+i'''}}v_k.
\end{align*}
Thus $T\cong \oplus_{r=0}T(i''+i''',r)$ and $U_s\cong U_{i''+i''',\alpha_s}$, where $j't^{i''}+j''t^s=\alpha_s t^{\beta_s}$ and $\alpha_s \in \{\beta^1, \dots , \beta^{m}\}$. Since the definition of $U_{i'',j''} \otimes U_{i''',j'''}$, we know $U_{i'',j''}\otimes U_{i''',j'''}\cong [\oplus_{r=0}T(i''+i''',r)]\oplus [\oplus_{s=0,s\neq i_0}^{q-1}U(i''+i''',\alpha_s)]$.

(6):\;$U_{i'',j''}.V_{(i''',j''',k'')}=\sum_{s=0}^{q-1}V_{(i''',i''+j''',t^{i''}(k''+j''t^{j'''+s}))}$. Suppose $U_{i'',j''}=\langle u_k|\;0\leq k\leq q-1 \rangle$ and $V_{(i''',j''',k'')}=\langle u_k'|\;0\leq k\leq q-1 \rangle$. Let $u_l^s=u_{s+i''-l}\otimes u_{l-i''}'$ and $V_s=\langle u_l^s|\;0\leq l\leq q-1 \rangle$, where $0\leq s\leq q-1$. Then we have
\begin{align*}
x.u_l^s=u_{l-1}^s, \; y.u_l^s=\eta^{[t^{i''}(k''+j''t^{j'''+s})]t^l}u_l^s,
\end{align*}
\begin{align*}
z_k.u_l^s=\delta_{k,i''+j'''}u_l^s,\;e_g.u_l^s=\delta_{g,a^{i'''}b^{i''+j'''}}u_l^s.
\end{align*}
Thus $V_s\cong V_{(i''',i''+j''',t^{i''}(k''+j''t^{j'''+s}))}$, where $0\leq s\leq q-1$. Since the definition of $U_{i'',j''}\otimes V_{(i''',j''',k'')}$, we know $U_{i'',j''}\otimes V_{(i''',j''',k'')}\cong \oplus_{s=0}^{q-1}V_{(i''',i''+j''',t^{i''}(k''+j''t^{j'''+s}))}$.

(7):\;$U_{i'',j''}.W_{(i''',j''',k'')}=\sum_{s=0}^{q-1}W_{(i''+i''',i''+j''',s)}$. Suppose $U_{i'',j''}=\langle u_k|\;0\leq k\leq q-1 \rangle$ and $W_{(i''',j''',k'')}=\langle v_l|\;0\leq l\leq p-1 \rangle$. Let $v_l^s=y^l.(\sum_{r=0}^{q-1}\omega^{rs}u_r\otimes v_0)$ and $W_s=\langle u_l^s|\;0\leq l\leq p-1 \rangle$, where $0\leq s\leq q-1$. Then we have
\begin{align*}
x.v_l^s=\omega^{k''-s}v_{lt^{-1}}^s, \; y.v_l^s=v_{l+1}^s,
\end{align*}
\begin{align*}
z_k.u_l^s=\delta_{k,i''+j'''}u_l^s,\;e_g.u_l^s=\delta_{g,a^{l(t^{i''+j'''}-t^{i''+i'''})}b^{i''+i'''}}u_l^s.
\end{align*}
Thus $W_s\cong W_{(i''+i''',i''+j''',k''-s)}$, where $0\leq s\leq q-1$. Since the definition of $U_{i'',j''}\otimes W_{(i''',j''',k'')}$, we know $U_{i'',j''}\otimes W_{(i''',j''',k'')}\cong \oplus_{s=0}^{q-1}W_{(i''+i''',i''+j''',s)}$.

(8.1):\; $V_{(i'',j'',k')}.V_{(i''',j''',k'')}=\sum_{s=0}^{q-1}V_{(\alpha_s',j''+j''',t^{\beta_s'}(k't^{j'''}+k''t^{-s}))}$ if $i''\notin -i'''\langle t\rangle$ and $i''+i'''t^{s+j''}=\alpha_s' t^{\beta_s'}$, where $\alpha_s'\in \{\beta^1\dots \beta^m\}$. Suppose $V_{(i'',j'',k')}=\langle u_k|\;0\leq k\leq q-1 \rangle$ and $V_{(i''',j''',k'')}=\langle u_k'|\;0\leq k\leq q-1 \rangle$. Let $u_k^s=x^{\beta_s'-k}.(u_0\otimes u_{s}')$ and $V_s=\langle u_k^s|\;0\leq k\leq q-1 \rangle$, where $0\leq s\leq q-1$. Then we have
\begin{align*}
x.u_k^s=u_{k-1}^s, \; y.u_k^s=\eta^{t^{\beta_s'}(k't^{j'''}+k''t^{-s})t^{-k}}u_k^s,
\end{align*}
\begin{align*}
z_l.u_k^s=\delta_{l,j''+j'''}u_k^s,\;e_g.u_k^s=\delta_{g,a^{\alpha_s' t^k}b^{j''+j'''}}u_k^s.
\end{align*}
Thus $V_s\cong V_{(\alpha_s',j''+j''',t^{\beta_s'}(k't^{j'''}+k''t^{-s}))}$, where $0\leq s\leq q-1$. Since the definition of $V_{(i'',j'',k')}\otimes V_{(i''',j''',k'')}$, we get $V_{(i'',j'',k')}\otimes V_{(i''',j''',k'')}\cong \oplus_{s=0}^{q-1}V_{(\alpha_s',j''+j''',t^{\beta_s'}(k't^{j'''}+k''t^{-s}))}$, where $\alpha_s', \beta_s'$ is determined by
$i''+i'''t^{s+j''}=\alpha_s' t^{\beta_s'}$ and $\alpha_s'\in \{\beta^1\dots \beta^m\}$.

(8.2):\;$V_{(i'',j'',k')}.V_{(i''',j''',k'')}=\sum_{s=0,s\neq s_0}^{q-1}V_{(\alpha_s',j''+j''',t^{\beta_s'}(k't^{j'''}+k''t^{-s}))}+\sum_{s=0}^{q-1}T_{j''+j''',s}$ if $i''+i'''t^{s_0+j''}=0$ for some $0\leq s_0\leq q-1$ and $k't^{j'''}+k''t^{-s_0}=0$. Suppose $V_{(i'',j'',k')}=\langle u_k|\;0\leq k\leq q-1 \rangle$ and $V_{(i''',j''',k'')}=\langle u_k'|\;0\leq k\leq q-1 \rangle$. Let $u_k^s=x^{\beta_s'-k}.(u_0\otimes u_{s}')$ and $V_s=\langle u_k^s|\;0\leq k\leq q-1 \rangle$, where $0\leq s\leq q-1$. Similar to before, we have
\begin{align*}
x.u_k^s=u_{k-1}^s, \; y.u_k^s=\eta^{t^{\beta_s'}(k't^{j'''}+k''t^{-s})t^{-k}}u_k^s,
\end{align*}
\begin{align*}
z_l.u_k^s=\delta_{l,j''+j'''}u_k^s,\;e_g.u_k^s=\delta_{g,a^{\alpha_s' t^k}b^{j''+j'''}}u_k^s.
\end{align*}
If $s\neq s_0$, then $V_s\cong V_{(\alpha_s',j''+j''',t^{\beta_s'}(k't^{j'''}+k''t^{-s}))}$, where $0\leq s\leq q-1$. For the case $s=s_0$, we define $v_k=\sum_{r=0}^{q-1}\omega^{kr}u_r^{s_0}$ and $T=\langle v_k|\;0\leq k\leq q-1 \rangle$. Since $k't^{j'''}+k''t^{-s_0}=0$, one can get that
\begin{align*}
x.v_k=\omega^k v_k, \; y.v_k=v_k,
\end{align*}
\begin{align*}
z_l.u_k^s=\delta_{l,j''+j'''}v_k,\;e_g.v_k=\delta_{g,b^{j''+j'''}}v_k.
\end{align*}
Thus $T\cong \oplus_{s=0}^{q-1}T_{j''+j''',s}$. Since the definition of $V_{(i'',j'',k')}\otimes V_{(i''',j''',k'')}$, we get $V_{(i'',j'',k')}.V_{(i''',j''',k'')}=\oplus_{s=0,s\neq s_0}^{q-1}V_{(\alpha_s',j''+j''',t^{\beta_s'}(k't^{j'''}+k''t^{-s}))}\oplus (\oplus_{s=0}^{q-1}T_{j''+j''',s})$.

(8.3):\;$V_{(i'',j'',k')}.V_{(i''',j''',k'')}=\sum_{s=0,s\neq s_0}^{q-1}V_{(\alpha_s',j''+j''',t^{\beta_s'}(k't^{j'''}+k''t^{-s}))}+U_{j''+j''',\alpha_{s_0}''}$ if $i''+i'''t^{s_0+j''}=0$ for some $0\leq s_0\leq q-1$ and $k't^{j'''}+k''t^{-s_0}=\alpha_{s_0}'' t^{\beta_{s_0}''}$, where $\alpha_s''\in \{\beta^1\dots \beta^m\}$. Suppose $V_{(i'',j'',k')}=\langle u_k|\;0\leq k\leq q-1 \rangle$ and $V_{(i''',j''',k'')}=\langle u_k'|\;0\leq k\leq q-1 \rangle$. Let $u_k^s=x^{\beta_s'-k}.(u_0\otimes u_{s}')$ and $V_s=\langle u_k^s|\;0\leq k\leq q-1 \rangle$, where $0\leq s\leq q-1$. Similar to before, we have
\begin{align*}
x.u_k^s=u_{k-1}^s, \; y.u_k^s=\eta^{t^{\beta_s'}(k't^{j'''}+k''t^{-s})t^{-k}}u_k^s,
\end{align*}
\begin{align*}
z_l.u_k^s=\delta_{l,j''+j'''}u_k^s,\;e_g.u_k^s=\delta_{g,a^{\alpha_s' t^k}b^{j''+j'''}}u_k^s.
\end{align*}
If $s\neq s_0$, then $V_s\cong V_{(\alpha_s',j''+j''',t^{\beta_s'}(k't^{j'''}+k''t^{-s}))}$, where $0\leq s\leq q-1$. For the case $s=s_0$, we define $\overline{u_k}=u_{\beta_s''-k}^{s_0}$ and $U=\langle \overline{u_k}|\;0\leq k\leq q-1 \rangle$. Then we have
\begin{align*}
x.\overline{u_k}=\overline{u_{k+1}}, \; y.\overline{u_k}=\eta^{\alpha_s''t^k}\overline{u_k},
\end{align*}
\begin{align*}
z_l.\overline{u_k}=\delta_{l,j''+j'''}\overline{u_k},\;e_g.\overline{u_k}=\delta_{g,b^{j''+j'''}}\overline{u_k}.
\end{align*}
Thus $U\cong U_{j''+j''',\alpha_{s_0}''}$. By definition of $V_{(i'',j'',k')}\otimes V_{(i''',j''',k'')}$, we have
$$V_{(i'',j'',k')}.V_{(i''',j''',k'')}=\oplus_{s=0,s\neq s_0}^{q-1}V_{(\alpha_s',j''+j''',t^{\beta_s'}(k't^{j'''}+k''t^{-s}))}\oplus U_{j''+j''',\alpha_{s_0}''}.$$

(9):\;$V_{(i'',j'',k')}.W_{(i''',j''',k'')}=\sum_{s=0}^{q-1}W_{(j''+i''',j''+j''',s)}$. Suppose $V_{(i'',j'',k')}=\langle u_k|\;0\leq k\leq q-1 \rangle$ and $W_{(i''',j''',k'')}=\langle v_l|\;0\leq l\leq p-1 \rangle$. For convenience, we denote $l_0=-i''d^{-1}t^{j''}$, where $d=t^{j'''}-t^{i'''}$. Let $v_l^s=y^l.(\sum_{j=0}^{q-1}\omega^{js}u_j\otimes v_{l_0t^j})$ and $W_s=\langle v_l^s|\;0\leq l\leq p-1 \rangle$, where $0\leq s\leq q-1$. Then we have
\begin{align*}
x.v_l^s=\omega^{k''+s}v_l^s, \; y.v_l^s=v_{l+1}^s,
\end{align*}
\begin{align*}
z_k.v_l^s=\delta_{k,j''+j'''}v_l^s,\;e_g.v_l^s=\delta_{g,a^{l(t^{j''+j'''}-t^{j''+i'''})}b^{j''+i'''}}v_l^s.
\end{align*}
Thus $W_s\cong W_{(j''+i''',j''+j''',k''+s)}$, where $0\leq s\leq q-1$. By definition of $V_{(i'',j'',k')}\otimes W_{(i''',j''',k'')}$, we get $V_{(i'',j'',k')}.W_{(i''',j''',k'')}=\oplus_{s=0}^{q-1}W_{(j''+i''',j''+j''',s)}$.

(10.1):\; $W_{(i,j,k)}.W_{(i',j',k')}=\sum_{s=0}^{q-1}mW_{(i+i',j+j',s)}+W_{(i+i',j+j',k+k')}$ if $i+i'\neq j+j'$. Suppose $W_{(i,j,k)}=\langle v_l|\;0\leq l\leq p-1 \rangle$ and $W_{(i',j',k')}=\langle v_l'|\;0\leq l\leq p-1 \rangle$. Define $M=\{(l,r)|(l,r)\in \mathbb{Z}_p\times \mathbb{Z}_q \;\text{such that }ld+rd't^i=0\}$, where $d=t^j- t^i$ and $d'=t^{j'}- t^{i'}$. Assume $(0,0)\neq (l_0,r_0)\in M$ and $0\leq s_0\leq q-1$. Let $\overline{v_l}=y^l.[\sum_{r=0}^{q-1}\omega^{s_0r}x^r.(v_{l_0}\otimes v_{r_0}')]$ and $W_{(l_0,s_0)}=\langle \overline{v_l}|\;0\leq l\leq p-1 \rangle$.
Then we have
\begin{align*}
x.\overline{v_l}=\omega^{-s_0}\overline{v_{t^{-1}l}}, \; y.\overline{v_l}=\overline{v_{l+1}},
\end{align*}
\begin{align*}
z_s.\overline{v_l}=\delta_{s,j+j'}\overline{v_l},\;e_g.\overline{v_l}=\delta_{g,a^{l(t^{j+j'}-t^{i+i'})}b^{i+i'}}\overline{v_l}.
\end{align*}
Thus $W_{(l_0,s_0)}\cong W_{(i+i',j+j',-s_0)}$. For the case $(l,r)=(0,0)\in M$, we define $v_l'=y^l.(v_0\otimes v_0')$ for $0\leq l\leq p-1$ and $W=\langle v_l'|\;0\leq l\leq p-1 \rangle$. Then we have
\begin{align*}
x.v_l'=\omega^{k+k'}v_{t^{-1}l}', \; y.v_l'=v_{l+1}',
\end{align*}
\begin{align*}
z_s.v_l'=\delta_{s,j+j'}v_l',\;e_g.v_l'=\delta_{g,a^{l(t^{j+j'}-t^{i+i'})}b^{i+i'}}v_l'.
\end{align*}
Thus $W\cong W_{(i+i',j+j',k+k')}$. Note that $\mathbb{Z}_p^\times=\cup_{i=1}^m \beta^i \langle t\rangle$ and since the definition of $W_{(i,j,k)}\otimes W_{(i',j',k')}$, we get $W_{(i,j,k)}\otimes W_{(i',j',k')}=\oplus_{s=0}^{q-1}W_{(i+i',j+j',s)}^{\oplus m}\oplus W_{(i+i',j+j',k+k')}$.

(10.2):\; $W_{(i,j,k)}.W_{(i',j',k')}=T_{i+i',k+k'}+\sum_{s=1}^m U_{i+i',\beta^s}+\sum_{l=0}^{p-1}\sum_{s=0}^{q-1}V_{(l,i+i',s)}$ if $i+i'=j+j'$. Suppose $W_{(i,j,k)}=\langle v_l|\;0\leq l\leq p-1 \rangle$ and $W_{(i',j',k')}=\langle v_l'|\;0\leq l\leq p-1 \rangle$. Let $u_0=\sum_{s=0}^{p-1}y^s.(v_0\otimes v_0')$ and $T=\langle u_0 \rangle$. Then we have
\begin{align*}
x.u_0=\omega^{k+k'}u_0, \; y.u_0=u_0,
\end{align*}
\begin{align*}
z_s.u_0=\delta_{s,i+i'}u_0,\;e_g.u_0=\delta_{g,b^{i+i'}}u_0.
\end{align*}
Thus $T\cong T(i+i',k+k')$. Let $1\leq s\leq p-1$. Define $v_r^s=x^r.(\sum_{l=0}^{p-1}\eta^{-ls}v_0\otimes v_0')$ and $U_s=\langle v_r^s|\;0\leq r\leq q-1 \rangle$. Then we have
\begin{align*}
x.v_r^s=v_{r+1}^s, \; y.v_r^s=\eta^{st^r}v_r^s,
\end{align*}
\begin{align*}
z_l.v_r^s=\delta_{l,i+i'}v_r^s,\;e_g.v_r^s=\delta_{g,b^{i+i'}}v_r^s.
\end{align*}
Thus $U_s\cong U_{i+i',s}$. Define $M=\{(l,r)|(l,r)\in \mathbb{Z}_p\times \mathbb{Z}_q \;\text{such that }ld+rd't^i=0\}$, where $d=t^j- t^i$ and $d'=t^{j'}- t^{i'}$. Let $(l_1,r_1)\in \mathbb{Z}_p\times \mathbb{Z}_q$ and $(l_1,r_1)\notin M$. Then $l_1d+r_1d't^i\neq0$ by definition. For simple, we define $f(l_1,r_1)=l_1d+r_1d't^i$. Let $\overline{v_l^s}=x^{-l}.(\sum_{r=0}\eta^{-rs}y^r.(v_{l_1}\otimes v_{r_1})')$ and let $V_{(s,f(l_1,r_1))}=\langle \overline{v_l^s}|\;0\leq l\leq q-1 \rangle$, where $0\leq s\leq p-1$ and $0\leq l\leq q-1$. Then we have
\begin{align*}
x.\overline{v_l^s}=\overline{v_{l-1}^s}, \; y.\overline{v_{l}^s}=\eta^{s.t^{-l}}\overline{v_{l}^s},
\end{align*}
\begin{align*}
z_r.\overline{v_{l}^s}=\delta_{r,i+i'}\overline{v_{l}^s},\;e_g.\overline{v_{l}^s}=\delta_{g,a^{f(l_1,r_1)t^l}b^{i+i'}}\overline{v_{l}^s}.
\end{align*}
Thus $V_{(s,f(l_1,r_1))}\cong V_{(f(l_1,r_1),i+i',s)}$. Using the definition, one can get $U_s=U_{ts}$ and $V_{(s,f(l_1,r_1))}=V_{(s,f(tl_1,tr_1))}$ for $0\leq s\leq p-1$ and $(l_1,r_1)\notin M$. Since the definition of $W_{(i,j,k)}\otimes W_{(i',j',k')}$, we get $W_{(i,j,k)}\otimes W_{(i',j',k')}=T_{i+i',k+k'}\oplus_{r=1}^{m}U_{i+i',\beta^r}\oplus (\oplus_{s=0}^{p-1}\oplus_{r=1}^m V(\beta^r,i+i',s))$.
\end{proof}

\textbf{Proof of Theorem \ref{thm3.x}.} Since Lemmas \ref{lem3.4.x}-\ref{lem3.5.x}, we get what we want.
\qed


\end{document}